\documentclass{amsart}
\usepackage{mathtools}
\usepackage{amsthm}
\usepackage[dvipsnames]{xcolor}
\usepackage{enumerate}
\usepackage{hyperref}
\hypersetup{
    unicode=false,          		
    pdftoolbar=true,        		
    pdfmenubar=true,        		
    pdffitwindow=false,     		
    pdfstartview={FitH},    		
    linktocpage=true,			    
    pdfnewwindow=true,      		
    colorlinks=true,       		    
    linkcolor=blue,          	    
    citecolor=PineGreen,    		
    filecolor=magenta,      		
    urlcolor=cyan           		
}
\theoremstyle{plain}
\newtheorem{theorem}{Theorem}[section]
\newtheorem{lemma}[theorem]{Lemma}

\newcommand{\diff}{\partial}
\newcommand{\dbar}{\overline{\partial}}
\newcommand{\kbar}{\overline{k}}
\newcommand{\ubar}{\overline{u}}

\newcommand{\zetabar}{\overline{\zeta}}

\newcommand{\calF}{\mathcal{F}}
\newcommand{\calI}{\mathcal{I}}

\newcommand{\calR}{\mathcal{R}}
\newcommand{\calS}{\mathcal{S}}

\newcommand{\R}{\mathbb{R}}
\newcommand{\C}{\mathbb{C}}

\newcommand{\rbar}{\overline{r}}
\newcommand{\zbar}{\overline{z}}

\newcommand{\dotarg}{\, \cdot \,}

\newcommand{\wu}{\widehat{u}}

\title[mNV and the Inverse Scattering Transform]{The Modified Novikov-Veselov Equation and the Inverse Scattering Transform}
\author{Peter Perry}
\address{Department of Mathematics, University of Kentucky, Lexington, Kentucky 40506--0027}
\email{pperr0@uky.edu}
\date{\today}

\allowdisplaybreaks

\numberwithin{equation}{section}

\begin{document}

\begin{abstract}
    This paper corrects several errors in the author's previous papers \cite{Perry2014} and \cite{Perry2016} respectively on the modified Novikov-Veselov (mNV) and the defocussing Davey-Stewartson II (DS II) equations. The mNV equation lies in the integrable hiearchy of the DS II equation, so that the same scattering transform can be used to solve both equations. 
    
    Each paper presents a proof that the solution formula
    $$ u(t) = \calI \left( e^{it\varphi} \calR (u_0) \right)$$
    (where $\calR$ (resp.\ $\calI$) is the direct (resp.\ inverse) scattering map, and $\varphi$ is a phase function) gives a solution to the PDE provided $u_0 \in \calS(\R^2)$. In \cite{Perry2014}, this results in an incorrect form of the mNV equation which has also been used by others who cited this paper. Here we correct errors in proofs and obtain a correct statement of the mNV equation as solved by inverse scattering. This result plays in important role in \cite{NPT:2025}. 
\end{abstract}
\maketitle
\tableofcontents

\section{Introduction}

The scattering transform for the defocussing Davey-Stewartson II (DS II) equation was studied by Perry in \cite{Perry2016} and applied in \cite{Perry2014} to the modified Novikov-Veselov(mNV) equation. A considerably better and more recent study of the scattering transform and its application may be found in \cite{NRT2020}, but here we will follow the notations and conventions of \cite{Perry2016}. Our purpose here is to correct several errors in \cite{Perry2014} and \cite{Perry2016} which, although they do not invalidate the final results in those papers, they gave rise to an incorrect formula for the mNV equation (see equation (1-10) in \cite{Perry2014}). The mNV equation as solved by inverse scattering in \cite{Perry2014} should read
\begin{align}
    \label{the.real.mNV}
        u_t + (\partial^3  + \dbar^3) u &=  N_{mNV}(u), \\
    \intertext{where}
    \label{the.real.N-mNV}
    \frac43 N_{mNV}(u) 
		&=  u \dbar^{-1} (\diff(\ubar \diff u)) 
                +  \diff u \cdot \dbar^{-1} (\diff (|u|^2)) \\
        \nonumber
        &\quad + u \diff^{-1}(\dbar(\ubar \dbar u)) 
                +  (\dbar u) \cdot \diff^{-1} (\dbar (|u|^2))
\end{align}

Here we will give a correct and complete derivation of the equations satisfied by the inverse scattering solution to the modified Novikov-Veselov equation. In section \ref{sec:inverse-scattering}, we recall the direct ($\calR$) and inverse ($\calI$) scattering maps for the DS II hierarchy. We refer the interested reader to the papers of Beals and Coifman \cite{BC1984,BC1989,BC1988}, the paper of \cite{Perry2016}, and the Fields lectures \cite{Perry2019} for further details on the transforms considered here, and to the paper \cite{NRT2020} of Nachman, Regev, and Tataru for results on the scattering transform used in \cite{NPT:2025} to study the mNV equation.  In section \ref{sec:evolution}, we derive an evolution equation for functions of the form
$$ u(t) = \calI \left( e^{it\varphi} (\calR u_0) \right)$$
where $u_0 \in \calS(\R^2)$ and $\varphi$ is a real-valued phase function.
In section \ref{sec:DSII-mNV}, we specialize to the phase functions \eqref{varphi.DS} and \eqref{varphi.mNV} for the DS and mNV equations. In an appendix, we give asymptotic expansions for the scattering solutions which are need for the computations in section \ref{sec:DSII-mNV}.

\subsection*{Acknowledgements}

The author thanks Daniel Tataru for pointing out the error in \cite{Perry2014} and thanks Adrian Nachman and Daniel Tataru for helpful discussions.

\section{Inverse Scattering}
\label{sec:inverse-scattering}

In what follows, we let
$$ e_k(z) = e^{\kbar \zbar - k z}.$$
We denote by $\diff$ and $\dbar$ the differential operators
$$ \diff f = \frac12 (\diff_x - i \diff_y) f, \quad
    \dbar f = \frac12 (\diff_x + i \diff y)f. $$
The operators $\diff_z$, $\dbar_z$ act on the $z=x+iy$ variable, while $\diff_k$ and $\dbar_k$ act on the complex $k$-variable. 
We will denote by $\dbar^{-1}$ and $\diff^{-1}$ the integral operators
\begin{align*}
    (\dbar^{-1} f)(z)
        &=  \frac{1}{\pi} \int_{\C} \frac{f(\zeta)}{z-\zeta} \, dA(\zeta),\\
    (\diff^{-1} f)(z)
        &=  \frac{1}{\pi} \int_{\C} \frac{f(\zeta)}{\zbar - \zetabar} \, dA(z)
\end{align*}

The direct scattering transform is based on the following $\dbar$-problem:
\begin{equation}
    \label{DSII.direct}
        \begin{aligned}
            \dbar_z \mu_1   &=  \frac12 e_k u (\overline{\mu_2)},\\
            \dbar_z \mu_2   &=  \frac12 e_k u (\overline{\mu_1)},\\
            (\mu_1,\mu_2)   &\underset{|z| \to \infty}{\sim} (1,0).
        \end{aligned}
\end{equation}
The \emph{direct scattering transform} of $u$ is given by
\begin{equation}
    \label{DSII.direct.map}
        \calR(u) = \frac{1}{\pi} \int e_{k}(z) u(z) \mu_1(z,k) \, dA(z),
\end{equation}
and its linearization at $u=0$ is
\begin{equation}
    \label{F.direct}
        (\calF f)(k) = \frac{1}{\pi} \int e_{k}(z) f(z) \, dA(z).
\end{equation}
The \emph{inverse scattering transform} is determined by the $\dbar$-problem
\begin{equation}
    \label{DSII.inverse}
        \begin{aligned}
            \dbar_k \nu_1   &=  \frac12 e_k \rbar (\overline{\nu_2}),\\
            \dbar_k \nu_2   &=  \frac12 e_k \rbar (\overline{\nu_1}),\\
            (\nu_1,\nu_2)   &\underset{|k| \to \infty}{\sim} (1,0).
        \end{aligned}
\end{equation}
The \emph{inverse scattering transform} of $r$ is given by 
\begin{equation}
    \label{DSII.inverse.map}
        u(z)    =   \frac{1}{\pi} \int e_{-k}(z) r(k) \nu_1(z,k) \, dA(k)
\end{equation}
and its linearization at $r=0$ is
\begin{equation}
    \label{F.inverse}
        (\calF^{-1} g)(z)    =   \frac{1}{\pi} \int e_{-k}(z) g(k) \, dA(k).
\end{equation}
The maps $\calR$, $\calI$ and their linearizations are continuous from $\calS(\R^2)$ to itself.

The DSII and mNV equations take the respective forms
\begin{align}
    \label{DSII}
    u_t + i(\diff^2_z + \dbar^2_z) u    &=   NL_{DS}(u)\\
    \intertext{and}
    \label{mNV}
    u_t + (\diff_z^3 + \dbar_z^3) u     &=  NL_{mNV}(u)
\end{align}
where $NL_{DS}(u)$ is a quadratic nonlinearity and $NL_{mNV}(u)$ is a cubic nonlinearity. To compute the dispersion relations for \eqref{DSII} and \eqref{mNV}, we take the $\calF_k$-Fourier transform of the linear equation using the relations
$$ \calF_k (\diff u) = -k \calF_k(u), \quad \calF_k (\dbar u) = \kbar \calF_k(u)$$
to obtain (setting $\widehat{u} = \calF_k(u)$)
\begin{align*}
    \wu_t + i(k^2 + \kbar^2) \wu &= 0   &\text{(DS II)},\\
    \wu_t + (-k^3 + \kbar^3 )             \wu &= 0   &\text{(mNV)}.
\end{align*}
We conclude that
\begin{equation}
    \label{varphi.DS}
        \varphi_{DS}(k) = k^2 + \kbar^2
\end{equation}
and
\begin{equation}
    \label{varphi.mNV}
        \varphi_{mNV}(k) = (-i)( \kbar^3 - k^3).
\end{equation}
This implies the solution formulas
\begin{align}
    \label{IST.DS}
        u(t)    &=  \calI \left( e^{it(k^2+\kbar^2)} \calR(u_0)\right),
                &\text{(DS II)},\\
    \label{IST.mNV}
        u(t)    &=  \calI \left( e^{it( ik^3-i\kbar^3)} \calR(u_0) \right),
                &\text{(mNV)}.
\end{align}

We will compute the evolution of \eqref{IST.mNV} using ideas of Beals and Coifman \cite{BC1984,BC1989,BC1988}. A similar computation for DS II is carried out in Sung \cite{SungIII:1994}.

\section{Evolution Equation for a Real-Valued Phase Function}
\label{sec:evolution}

In this section we consider evolutions of the form
\begin{equation}
    \label{phase.moves}
    u(t) = \calI \left( e^{it\varphi} \calR (u_0) \right)
\end{equation}
where $u_0 \in \calS(\R^2)$ and the maps $\calR$ and $\calI$ are as defined in \eqref{DSII.direct.map} and \eqref{DSII.inverse.map}. We will also need the auxiliary problem
\begin{equation}
    \label{DSII.dual}
        \begin{aligned}
        \dbar_k \nu_1^\sharp &= \frac12 e_k r (\overline{\nu_2^\sharp}),\\
        \dbar_k \nu_2^\sharp &= \frac12 e_k r (\overline{\nu_1^\sharp}),\\
        (\nu_1^\sharp,\nu_2^\sharp) &\underset{|k| \to \infty}{\sim} (1,0).
        \end{aligned}
\end{equation}
To establish the connection between the problems \eqref{DSII.dual} and \eqref{DSII.direct}, we will need the Lemma B.1 from \cite{Perry2016}, which we restate here to correct a typo in part (iii). In \cite{Perry2016} we consider $u \in H^{1,1}(\C)$, but here we only need the result for $u \in \calS(\R^2)$. We recall that $\calI$ and $\calR$ both map $\calS(\R^2)$ to itself (see, for example, the papers \cite{BC1984,BC1989,BC1988} of Beals and Coifman).

\begin{lemma}
    \label{lemma:R.map}
    Let $u, u^\flat \in \calS(\R^2)$ and let $r = \calR(u)$, $r^\flat = \calR(u^\flat)$. 
    \begin{enumerate}[(i)]
        \item   If $u^\flat(z) = -u(z)$, then $r^\flat(k) = -r(k)$
        \item   If $u^\flat(z)=-u(-z)$, then $r^\flat(k) = -r(-k).$
        \item   If $u^\flat(z) = \overline{u(z)}$, then $r^\flat(k) =  \overline{r(-k)}$.
    \end{enumerate}
\end{lemma}

As a consequence of Lemma \ref{lemma:R.map}, it follows that the system \eqref{DSII.dual}
corresponds to a potential 
$$ u^\sharp(z) = \overline{u(z)}.$$

Here and in what follows, we will use the shorthand
$$ \langle f, g \rangle = \frac{1}{\pi} \int \overline{f(k)} g(k) \, dA(k). $$

For $u_0 \in \calS(\R^2)$, $u(t)$ as given in \eqref{phase.moves}, we have (see \eqref{DSII.inverse.map})
$$
u(t) = \langle e_k \rbar, \nu_1 \rangle 
$$
so that 
\begin{align*}
        \dot{u}(t) 
            &= \langle e_k \dot{\rbar}, \nu_1 \rangle + 
                \langle e_k \rbar, \dot{\nu_1} \rangle, \\
            &=  i \langle e_k \varphi \rbar, \nu_1 \rangle +
                \langle e_k \rbar, \dot{\nu_1} \rangle
\end{align*}
To compute $\dot{\nu_1}$, we need a solution formula for $\nu_1$. From the system
\eqref{DSII.inverse}, it is easy to see that
\begin{align*}
    \nu_1   &=  1   +   
                \frac12 \dbar_k^{-1} \left( e_k \rbar \overline{\nu_2}\right),\\
    \nu_2   &=  \frac12 \dbar_k^{-1} \left( e_k \rbar \overline{\nu_1}\right).
\end{align*}
Letting
$$ T_k f = \frac14 \dbar_k^{-1} \left( e_k \rbar \diff_k^{-1} (e_{-k} r f) \right),$$
it follows that
\begin{equation}
    \label{DSII.inverse.sol}
        \begin{aligned}
        \nu_1   &=  (I-T_k)^{-1} 1, \\
        \nu_2   &=  \frac12 \dbar_k^{-1} \left( e_k \rbar \overline{\nu_1} \right)
        \end{aligned}
\end{equation}
We may then compute
\begin{align*}
    \dot{\nu}_1   &=  \left[ \diff_t, (I-T_k)^{-1} \right] 1\\
            &=  (I-T_k)^{-1} \left[ \diff_t, T_k \right] (I-T)^{-1} 1\\
            &=  (I-T_k)^{-1} \left[ \diff_t, T_k \right] \nu_1
\end{align*}
Since
\begin{align*}
    \left[\diff_t, T_k \right] \nu_1
        &= 
            \frac14 \dbar_k^{-1} 
                \left( 
                    e_k \overline{\dot{r}} \diff_k^{-1} (e_{-k} r \nu_1) + 
                    e_k \overline{r} \diff_k^{-1} e_{-k} (\dot{r} \nu_1) 
                \right)\\
        &=
            -\frac{i}{4} \dbar_k^{-1}
                \left(
                    e_k \overline{r} 
                        \left[ \varphi, \diff_k^{-1}\right]
                        (e_{-k} r \nu_1)
                \right),
\end{align*}
we conclude that
\begin{align*}
    \dot{u}(t)
        &=  i\langle e_k \varphi \rbar, \nu_1 \rangle \\
        &\quad -
            \frac{i}{4}
                \left\langle
                    e_k \rbar,
                    (I-T_k)^{-1} 
                        \dbar_k^{-1}
                            \left( 
                                e_k \rbar \varphi \diff_k^{-1}
                                    \left(e_{-k} r \nu_1 \right)
                            \right)
                \right\rangle\\
        &\quad +
            \frac{i}{4}
                \left\langle
                    e_k \rbar,
                    (I-T_k)^{-1}
                        \dbar_k^{-1}
                            \left(
                                e_k \rbar \diff_k^{-1}
                                    \left( e_{-k} r \varphi \nu_1 \right)
                            \right)
                \right\rangle
\end{align*}
It now follows that
\begin{align}
        \label{how.ut.moves}
       \dot{u}(t) &= i\langle e_k \varphi \rbar, \nu_1 \rangle\\
            \nonumber
            &\quad +
                \frac{i}{4}
                    \left\langle
                        \diff_k^{-1} (I-T_k^*)^{-1}(e_k \rbar),
                        e_k \rbar \varphi \diff_k^{-1} ( e_{-k}  r \nu_1 )
                    \right\rangle\\
            \nonumber
            &\quad +
                \frac{i}{4}
                    \left\langle
                        \dbar_k^{-1}
                            \left(
                                e_{-k} r
                                \diff_k^{-1}
                                (I-T_k^*)^{-1}(e_k \rbar)
                            \right),
                            e_{-k} r \varphi \nu_1
                    \right\rangle
\end{align}
where
$$ T_k^* f = 
    \frac14 e_k \rbar \dbar_k^{-1} 
        \left(
            r e_{-k} \diff_k^{-1} f
        \right) 
$$

To simplify these expressions, we will need solution formulas for $\nu_1^\sharp(-z,k)$ and $\nu_2^\sharp(-z,k)$. These may be obtained from the equations (compare \eqref{DSII.dual})
\begin{equation}
    \label{DSII.dual.-z}
        \begin{aligned}
        \dbar_k \nu_1^\sharp(-z,\dotarg)  
            &=  \frac12 e_{-k} r \left( \, \overline{\nu_2^\sharp(-z,\dotarg)} \, \right),\\
        \dbar_k \nu_2^\sharp(-z,\dotarg)
            &=  \frac12 e_{-k} r \left( \, \overline{\nu_1^\sharp(-z,\dotarg)} \,\right)\\
        \left( \nu_1^\sharp(-z,k), \nu_2^\sharp(-z,k) \right)
            &\underset{|k| \to \infty}{\sim} (1,0)
        \end{aligned}
\end{equation}
Iterating these equations we have
\begin{align*}
    \nu_1^\sharp(-z,\dotarg)
        &=  1 + \frac14 \dbar_k^{-1} 
                    \left(
                        e_{-k} r \diff_k^{-1}( e_k \rbar \nu_1^\sharp(-z,\dotarg))
                    \right),\\
    \nu_2^\sharp(-z,\dotarg)
        &=  \frac12 \dbar_k^{-1} 
                    \left(  
                        e_{-k} r \overline{\nu_1^\sharp(-z,\dotarg)} 
                    \right).
\end{align*}
or
\begin{align}
    \label{nu1.sharp.sol.-z}
        \nu_1^\sharp(-z,\dotarg)
            &=  (I-T_k^\sharp)^{-1} 1,\\
    \label{nu2.sharp.sol.-z}
        \nu_2^\sharp(-z,\dotarg)
            &=  \frac12 
                    \dbar_k^{-1} 
                    \left( 
                        e_{-k} r \overline{(I-T_k^\sharp)^{-1} 1 }
                    \right)
\end{align}
where
\begin{equation}
    \label{Tk.sharp.-z}
        (T_k^\sharp f)(-z, \dotarg) = \frac14 \dbar_k^{-1}
            \left(
                e_{-k} r \diff_k^{-1} (e_k \rbar f(\dotarg) )
            \right).
\end{equation}

We now return to simplifying \eqref{how.ut.moves}. First, it follows from \eqref{nu1.sharp.sol.-z} and \eqref{nu2.sharp.sol.-z} that 
\begin{align}
    \label{u.move.line3}
    \frac14 \dbar_k^{-1} (e_{-k}r \diff_k^{-1}(I-T_k^*)^{-1} (e_k \rbar))
    &=  T_k^\sharp (I-T_k^\sharp)^{-1} (1) \\
    \nonumber
    &=  (\nu_1^\sharp(-z,\dotarg) - 1) \\
    \intertext{and}
    \label{u.move.line2}
    \frac12 \diff_k^{-1}(I-T_k^*)^{-1}(e_k \rbar)
    &=  \frac12 \diff_k^{-1} 
            \left( 
                e_k \rbar (I-T_k^\sharp)^{-1} 1
            \right)\\
    \nonumber
    &=  \frac12 \diff_k^{-1}
            \left(
                e_k \rbar \nu_1^\sharp(-z,\dotarg)
            \right)\\
    \nonumber
    &=    \overline{\nu_2^\sharp (-z,\dotarg)}
\end{align}
Using \eqref{u.move.line3} and \eqref{u.move.line2} in \eqref{how.ut.moves}, we conclude that
\begin{align}
    \label{the.way.u.moves}
    \dot{u}(t)
        &=
        i\left\langle \, 
            \overline{\nu_2^\sharp(-z,\dotarg)},\varphi 
                e_k \rbar
                \overline{\nu_2(z,\dotarg)}
                \,
            \right\rangle
        +
        i\left\langle 
                \nu_1^\sharp(-z,\dotarg), \varphi e_{-k}r \nu_1(z,\dotarg)
            \right\rangle.
\end{align}
Expressing \eqref{the.way.u.moves} as an integral we obtain
\begin{align}
\label{more.u.moves}
\dot{u}(t)
    &=  \frac{2i}{\pi}
            \int 
                \varphi(k) \eta(z,k)
            \, dA(k)
    \\
    \intertext{where}
\label{what.eta.is}
    \eta(z,k)
        &=   \nu_2^\sharp(-z,k) 
                \frac12 e_{k} \rbar(k) \overline{\nu_2(z,k)} +
            \frac12 e_{-k}(z) r(k)  \overline{\nu_1^\sharp(-z,k)} 
                 \nu_1(z,k)    
\end{align}
Using equations \eqref{DSII.inverse} and \eqref{DSII.dual.-z}, we have
\begin{align*}
    \dbar_k \left[\nu_2^\sharp(-z,k) \nu_1(z,k) \right]
        &=  
            \frac12 e_{-k} r 
                \overline{\nu_1^\sharp(-z,k)} 
                \nu_1(z,k) +
            \frac12 e_k \rbar \nu_2^\sharp(-z,k) 
                \overline{\nu_2(z,k)}\\
    \intertext{and}
    \dbar_k \left[ \nu_1^\sharp(-z,k) \nu_2(z,k)\right]
        &=  \frac12 e_{-k} r \overline{\nu_2^\sharp(z,-k)} 
            \nu_2(z,k) +
            \nu_1^\sharp(-z,k) \frac12 e_k \rbar \overline{\nu_1(z,k)}
\end{align*}
It follows that 
\begin{align}
    \label{eta.is}
    \eta(z,k)   &=  \dbar_k 
        \left[
            \nu_2^\sharp(-z,k) \nu_1(z,k)
        \right]
\end{align}
and
\begin{align}
    \label{etabar.is}
    \overline{\eta(z,k)}
        &=  \dbar_k
            \left[ 
                \nu_1^\sharp(-z,k) \nu_2(z,k)
            \right]
\end{align}
We will use \eqref{eta.is} and \eqref{etabar.is} in the next section to derive equations of motion for the evolutions \eqref{IST.DS} and \eqref{IST.mNV}.

\section{The mNV Equation}
\label{sec:DSII-mNV}

In this section we analyze the evolution \eqref{IST.mNV}. We will use the following lemma to compute integrals involving the quantities \eqref{eta.is} and \eqref{etabar.is}.

\begin{lemma}
    \label{lemma:g.expand}
    Suppose that $g$ is a $C^\infty$ function on $\C$ so that $\dbar_k g$ vanishes rapidly as $|k| \to \infty$, and $g$ admits an asymptotic expansion of the form
    \begin{equation}
        \label{g.asy}
            g(k,\kbar) \underset{|k| \to \infty}{\sim} a_0 + \sum_{\ell \geq 0} \frac{g_\ell}{k^{\ell+1}}
    \end{equation}
    for a constant $a_0$. For any $n \geq 1$,
    \begin{equation}
        \label{kn.g.asy}
                \frac{1}{\pi} \int_{|k| \leq R} k^n (\dbar_k g)(k) \, dA(k) = g_n.
    \end{equation}
\end{lemma}

\begin{proof}
    First, since $\dbar_k g$ has rapid decay as $|k| \to \infty$, we may compute
    $$
        \frac{1}{\pi} \int k^n (\dbar_k g)(k) \, dA(k) =
        \lim_{|k| \to \infty}
                \frac{1}{\pi} \int_{|k| \leq R} k^n (\dbar_k g)(k) \, dA(k
    $$
    Next, since $k^n$ is holomorphic we may write
    $$ \int_{|k| \leq R} k^n (\dbar_k g)(k) \, dA(k) =
        \int_{|k| \leq R} \dbar_k (k^n g)(k) \, dA(k). $$
    By the Gauss-Green theorem, we have
    $$ \int_{|k| \leq R} \dbar_k (k^n g)(k) \, dA(k) =
        \int_{|k|=R} k^n g(k) \frac{k}{2|k|} \, ds$$
    where  $ds$ is surface measure. The result follows by using the asymptotic expansion of $g$ and using the fact that 
    $$\dfrac{1}{2\pi} \displaystyle\int_{|k|=R} k^{m-1} \frac{k}{|k|} \, ds = \delta_{m0}. $$
\end{proof}

Using the phase function \eqref{varphi.mNV} in \eqref{more.u.moves}, we have
\begin{align}
    \label{how.mNV.moves.1}
    \dot{u} &=  \frac{2i}{\pi}
                    \int    (-i\kbar^3 + ik^3 )\eta(z,k) \, dA(k) \\
            \nonumber
            &=  \frac{2i}{\pi}
                    \int 
                        \left[ 
                            ik^3 \eta(z,k) + 
                            \overline{
                                \, \, ik^3 \overline{\eta(z,k) \,\, }
                            }
                        \right],
                    \, dA(k)\\
    \nonumber
            &=  2(I_1 + \overline{I_2})
\end{align}
where
\begin{align}
    \label{u.mNV.I1}
        I_1 &=  -\frac{1}{\pi} 
                    \int k^3 \dbar_k
                        \left[
                            \nu_2^\sharp(-z,k) \nu_1(z,k)
                        \right]
                    \, dA(k),\\
    \label{u.mNV.I2}
        I_2 &=  \hphantom{-}
                \frac{1}{\pi} 
                    \int k^3 \dbar_k
                    \left[ 
                        \nu_1^\sharp(-z,k) \nu_2(z,k)
                    \right]
                    \, dA(k).
\end{align}
To obtain \eqref{u.mNV.I1} and \eqref{u.mNV.I2}, we used the identities \eqref{eta.is} and \eqref{etabar.is}. 

As shown in Appendix \ref{sec:expand}, the scattering solutions $\nu_1$, $\nu_2$, $\nu_1^\sharp$, $\nu_2^\sharp$ have large-$k$ asymptotic expansions of the form \eqref{g.asy}, so the same is true of their products. This means that we can use Lemma \ref{lemma:g.expand} to evaluate \eqref{u.mNV.I1} and \eqref{u.mNV.I2}. We obtain 
\begin{align}
    \label{u.mNV.I1.bis}
    I_1 &=  - \left[ \nu_2^\sharp(-z,k) \nu_1(z,k)\right]_3,\\
    I_2 &=  \hphantom{-} \left[ \nu_1^\sharp(-z,k) \nu_2(z,k)\right]_3
\end{align}
where, for a quantity with asymptotic expansion \eqref{g.asy}, $\left[ \dotarg \right]_\ell$ denotes $g_\ell$. These quantities can be expressed in terms of the large-$k$ asymptotic expansions for the scattering solutions. 

\subsection{\texorpdfstring{Evaluating $I_1$ and $I_2$}{Evaluating I1 and I2}}

Denoting
\begin{align*}
    \nu_1(z,k) &
        \underset{k \to \infty}{\sim} 
            1 + \sum_{\ell=0}^\infty \frac{\nu_{1,\ell}(z)}{k^{\ell+1}},\\
    \nu_2(z,k) &
        \underset{k \to \infty}{\sim}
            \sum_{\ell=0}^\infty \frac{\nu_{2,\ell}(z)}{k^{\ell+1}},\\
    \nu_1^\sharp(-z,k) &
         \underset{k \to \infty}{\sim}
            1 + \sum_{\ell=0}^\infty \frac{\nu_{1,\ell}^\sharp(z)}{k^{\ell+1}},\\
    \nu_2^\sharp(-z,k) &
        \underset{k \to \infty}{\sim}
            \sum_{\ell=0}^\infty \frac{\nu_{2,\ell}^\sharp(z)}{k^{\ell+1}}
\end{align*}
we have\footnote{To obtain, for example, \eqref{I1.expand},  we use the expansions to compute
\begin{align*}
    \nu_1 \nu_2^\sharp
        &\sim \sum_{\ell=0}^\infty
            \frac{\nu_{2,l}^\sharp}{k^{\ell+1}} +
            \sum_{r=0}^\infty 
                \sum_{s=0}^r
                    \frac{\nu_{1,s}\nu_{2,r-s}^\sharp }{k^{r+2}}
\end{align*}
and extract the coefficient of $k^{-4}$. The method for \eqref{I2.expand} is analogous. }
\begin{align}
    \label{I1.expand}
    \left[
        \nu_1(z,\diamond) \nu_2^\sharp(-z,\diamond)
    \right]_3
    &=  \nu_{2,3}^\sharp + 
        \nu_{2,2}^\sharp \nu_{1,0} +
        \nu^\sharp_{2,1} \nu_{1,1} +
        \nu_{2,0}^\sharp \nu_{1,2}
    \intertext{and}
    \label{I2.expand}
    \left[
        \nu_2(z,\diamond) \nu_1^\sharp(-z,\diamond)
    \right]_3
    &=  \nu_{2,3}  +
        \nu_{2,2} \nu^\sharp_{1,0} +
        \nu_{2,1} \nu^\sharp_{1,1} +
        \nu_{2,0} \nu^\sharp_{1,2}
\end{align}

\subsubsection{Evaluating $I_1$}
Using  \eqref{nus2.2}, \eqref{nu.start}, \eqref{nus12.1}, and \eqref{nu.12}, we obtain
\begin{align}
    \label{I1.expand.2}
        \nu_{2,2}^\sharp \nu_{1,0}
            &=  
                {+}\frac{1}{128}u \dbar^{-1}\left(|u|^2 \dbar^{-1}(|u|^2)\right) 
                        \cdot \dbar^{-1}(|u|^2) \\
            \nonumber
            &\quad
                { -} \frac{1}{32} \diff \left( u \dbar^{-1}(|u|^2) \right)
                        \cdot \dbar^{-1}(|u|^2)
                { - } \frac{1}{32} u \dbar^{-1} (\ubar \diff u) 
                        \cdot \dbar^{-1}(|u|^2)\\
            \nonumber
            &\quad
                {+} \frac18 \diff^2 u
                        \cdot \dbar^{-1}(|u|^2)\\[10pt]
    \label{I1.expand.3}
        \nu^\sharp_{2,1}\nu_{1,1}
            &=  \left({-}\frac18 u \dbar^{-1}(|u|^2) { +} \frac12 \diff u\right) \\
            \nonumber
            &\qquad \qquad
                \cdot
                \left(\frac{1}{16} \dbar^{-1}(|u|^2 \dbar^{-1}(|u|^2)) -\frac14 \dbar^{-1}(u \diff \ubar)\right),\\
            \nonumber
            &=  { -}\frac{1}{128} u \dbar^{-1}(|u|^2) \cdot 
                    \dbar^{-1}\left( |u|^2 \dbar^{-1}(|u|^2) \right)\\
            \nonumber
            &\qquad {+} \frac{1}{32} u \dbar^{-1}(|u|^2) \cdot
                    \dbar^{-1}(u \diff \ubar)  
                    { +} \frac{1}{32} \diff u \cdot 
                    \dbar^{-1} \left( |u|^2 \dbar^{-1}(|u|^2) \right)\\
            \nonumber
            &\qquad  {-} \frac18 \diff u \cdot \dbar^{-1} (u \diff \ubar) \\[10pt]
    \label{I1.expand.4}
        \nu_{2,0}^\sharp \nu_{1,2}
            &=  { +}\frac{1}{128}u
                        \dbar^{-1}
                            \left(
                                |u|^2 \dbar^{-1}
                                    \left(|u|^2 \dbar^{-1}(|u|^2) \right)
                            \right)\\[5pt]
                \nonumber
                &\qquad 
                    {-}\frac{1}{32} u
                        \left\{
                            \dbar^{-1}\left(u \diff (\ubar \dbar^{-1}(|u|^2)) \right) + 
                            \dbar^{-1}\left(|u|^2 \dbar^{-1} (u \diff \ubar ) \right) 
                        \right\}\\
                \nonumber
                &\qquad
                    {+}\frac18 u \dbar^{-1}(u \diff^2 \ubar )       
\end{align}
It is easy to see that the seventh-order terms cancel. The fifth-order terms also cancel: to see this, one uses the identity
$$ (\dbar^{-1} f) \cdot (\dbar^{-1} g) = \dbar^{-1} \left( f \dbar^{-1} g + g \dbar^{-1} f \right)$$
(a kind of `product rule' valid for $f$ and $g$ that vanish at infinity). Terms involving
$ \diff u \cdot \dbar^{-1} (|u|^2 \dbar^{-1}(|u|^2)$ cancel as do the remaining terms, which involve the four quantities 
\begin{align*}
u \dbar^{-1}(\ubar \diff u \dbar^{-1}(|u|^2)),
    &&  u \dbar^{-1}( u \diff \ubar \dbar^{-1}(|u|^2)),\\
u \dbar^{-1} (|u|^2 \dbar^{-1} (\ubar \diff u)), &&
    \quad   u \dbar^{-1}(|u|^2 \dbar^{-1}(u \diff \ubar)).
\end{align*}
The only nonzero terms come from the third-order terms in \eqref{nus.23}, \eqref{I1.expand.2}, \eqref{I1.expand.3}, \eqref{I1.expand.4}, i.e., 
\begin{align*}
\left[ \nu_{1,1}(z,\diamond)\nu_2^\sharp(-z,\diamond) \right]_3
    &=  { -}\frac18 \left[ u\dbar^{-1}
                \left(
                    \ubar \diff^2 u
                \right)
            + \diff^2
                \left(
                    u \dbar^{-1}(|u|^2)
                \right)
            + \diff 
                \left(
                    u \dbar^{-1}(\ubar \diff u )
                \right)\right] 
\\
    & \quad 
        { + } \frac18 \diff^2 u \cdot \dbar^{-1}(|u|^2) 
        { - } \frac18 \diff u \cdot \dbar^{-1}(u \diff \ubar)
        {+} \frac18 u \dbar^{1}(u \diff^2 \ubar )\\
    &\quad    {+} \frac12 \diff^3 u  
\end{align*}
Upon simplification, we find that
\begin{align}
    \label{I1.exp.final}
        \begin{aligned}
            \left[\nu_{1}(z,\diamond)\nu_2^\sharp(-z,\diamond) \right]_3 
                &=  -\frac38 u \dbar^{-1} (\ubar \diff^2 u)
                    -\frac38 \diff u \cdot \dbar^{-1}(u \diff \ubar)\\
                &\quad  
                    -\frac38 u \dbar^{-1}(\diff u \cdot \diff \ubar)
                    -\frac38 \diff u \cdot \dbar^{-1}(\ubar \diff u)\\
                &\quad  
                    + \frac12 \diff^3 u
        \end{aligned}
\end{align}

\subsubsection{Evaluating $I_2$}

We will focus on first- and third-order terms since the cancellations for seventh- and fifth-order terms are similar to those already considered.

Neglecting the seventh- and fifth-order terms we have
\begin{align}
    \label{I2.expand.1}
        \nu_{2,3}
            &\simeq \frac18 
                \left\{ 
                    \ubar \dbar^{-1}\left(u \diff^2 \ubar\right)
                    + \diff^2\left(\ubar \dbar^{-1}(|u|^2)\right) 
                    + \diff
                        \left(\ubar \dbar^{-1}(u \diff \ubar) \right)    
                \right\} \\
        \nonumber
        &\quad 
                -\frac12 \diff^3 \ubar\\
    \label{I2.expand.2}
        \nu_{2,2} \nu_{1,0}^\sharp  
            &\simeq -\frac18 \diff^2(\ubar) \cdot \dbar^{-1}(|u|^2)\\
    \label{I2.expand.3}
        \nu_{2,1} \nu_{1,1}^\sharp  
            &\simeq \frac18 \diff \ubar \cdot \dbar^{-1} (\ubar \diff u)\\
    \label{I2.expand4}
        \nu_{2,0}\nu_{1,2}^\sharp   
            &\simeq -\frac18 \dbar^{-1} (\ubar \diff^2 u) \cdot \ubar 
\end{align}
A short computation with \eqref{I2.expand.1}--\eqref{I2.expand4} shows that
\begin{equation}
    \label{I2.exp.final}
    \begin{aligned}
       \left[
        \nu_2(z,\diamond) \nu_1^\sharp(-z,\diamond) \right]_3
            &= \frac38 \diff \ubar \dbar^{-1}(u \diff \ubar) 
                + \frac38 \diff \ubar \dbar^{-1} (\ubar \diff u) \\
            &\quad
                + \frac38 \ubar \dbar^{-1}(\diff u \cdot \diff \ubar)
                + \frac38 \ubar \dbar^{-1}( u \diff^2 \ubar)
           \\
            &\quad
                -  \frac12 \diff^3 \ubar
    \end{aligned}
\end{equation}

\subsection{Equation of Motion}

Using \eqref{how.mNV.moves.1}, \eqref{u.mNV.I1}--\eqref{u.mNV.I2}, \eqref{I1.exp.final}, and \eqref{I2.exp.final}, we conclude that
\begin{equation*}
    \begin{aligned}
        u_t + \diff^3 u + \dbar^3 u
            &= \frac34 u \dbar^{-1} (\ubar \diff^2 u)
                    + \frac34 \diff u \cdot \dbar^{-1}(u \diff \ubar)\\
                &\quad  
                    + \frac34 u \dbar^{-1}(\diff u \cdot \diff \ubar)
                    + \frac34 \diff u \cdot \dbar^{-1}(\ubar \diff u)\\
                &\quad 
                + \frac34 \dbar u \cdot \diff^{-1}(\ubar \dbar u) 
                + \frac34 \dbar u \diff^{-1} (u \dbar \ubar) \\
            &\quad
                + \frac34 u \diff^{-1}(\dbar \ubar \cdot \dbar u)
                + \frac34 u \diff^{-1}( \ubar \dbar^2 u)
    \end{aligned}
\end{equation*}
or, combining similar terms
\begin{equation}
    \label{mNV.u.eqn}
        \begin{aligned}
        u_t + \diff^3 u + \dbar^3 u
            &=  \frac34 u \dbar^{-1} (\diff(\ubar \diff u)) 
                + \frac34 \diff u \cdot \dbar^{-1} (\diff (|u|^2)) \\
            &\quad + \frac34 u \diff^{-1}(\dbar(\ubar \dbar u)) 
                + \frac34 (\dbar u) \cdot \diff^{-1} (\dbar (|u|^2))
        \end{aligned}
\end{equation}

\appendix

\section{\texorpdfstring{Large-$k$ Asymptotics of Scattering Solutions}{Large-k Behavior of Scattering Solutions }}
\label{sec:expand}

In this section, we obtain large-$k$ asymptotics of solutions to \eqref{DSII.inverse} and the auxiliary problem \eqref{DSII.dual}, following the discussion in \cite[Appendix C]{Perry2016}. 

\subsection{\texorpdfstring{Expansions for $\nu_1$ and $\nu_2$}{Expansions for nu1 and nu2}}

We will use the fact that $\nu_1,\nu_2$ also satisfy the following equations in $z$:
\begin{equation}
    \label{nu12.z}
    \begin{aligned}
        \dbar_z \nu_1
            &=  \frac12 u \nu_2 \\
        (\diff_z + k) \nu_2
            &=  \frac12 \ubar \nu_1
    \end{aligned}
\end{equation}
We can find large-$k$ expansions of the form
\begin{align*}
    \nu_1(z,k)  &\underset{|k| \to \infty}{\sim}
        1 + \sum_{\ell=0}^\infty \frac{\nu_{1,\ell}}{k^{\ell+1}},\\
    \nu_2(z,k)  &\underset{|k| \to \infty}{\sim}
            \sum_{\ell=0}^\infty \frac{\nu_{2,\ell}}{k^{\ell+1}}
\end{align*}
using the initial data
\begin{equation} 
    \label{nu.start}
    \nu_{1,0} = \frac14 \dbar^{-1}(|u|^2), \quad
    \nu_{2,0} = \frac12 \ubar 
\end{equation}
and the recurrence relations
\begin{align}
    \label{nu2.recur}
        \nu_{2,\ell}    &=  \frac12 \ubar \nu_{1,\ell-1} - \diff \nu_{2,\ell-1}
    \\
    \label{nu1.recur}
        \nu_{1,\ell}    &=  \frac12 \dbar^{-1} (u \nu_{2,\ell})
\end{align}
true for $\ell \geq 1$. The results are:
\begin{align}
    \label{nu.11}
    \nu_{1,1}   &=  \frac{1}{16} \dbar^{-1} \left( |u|^2 \dbar^{-1} (|u|^2) \right) -
                    \frac14 \dbar^{-1} \left( u \diff \ubar \right) \\
    \label{nu.21}
    \nu_{2,1}   &=  \frac18 \ubar \dbar^{-1} (|u|^2) - \frac12 \diff \ubar\\
    \label{nu.12}
    \nu_{1,2}   &=  \frac{1}{64} \dbar^{-1} 
                        \left(
                            |u|^2 \dbar^{-1}
                                (
                                    |u|^2 \dbar^{-1} (|u|^2)
                                )
                        \right)\\
                \nonumber
                &\quad - \frac{1}{16} 
                          \left\{
                                \dbar^{-1}
                                    \left(
                                        u \diff 
                                            (\ubar \dbar^{-1} (|u|^2)
                                    \right) +
                                \dbar^{-1}
                                    \left(
                                        |u|^2 \dbar^{-1}
                                            (u \diff \ubar)
                                    \right)
                          \right\}
                        \\
                 \nonumber
                &\quad + \frac14 \dbar^{-1} (u \diff^2 \ubar)
                \\
    \label{nu.22}
    \nu_{2,2}   &=  \frac{1}{32} 
                        \ubar  \dbar^{-1} \left(|u|^2 \dbar^{-1} (|u|^2)\right)\\
                \nonumber
                &\quad 
                    -\frac18
                        \left\{
                            \diff 
                                \left( 
                                    \ubar \dbar^{-1} (|u|^2) 
                                \right) +
                            \ubar \dbar^{-1} 
                                \left( 
                                    u \diff \ubar
                                \right)
                        \right\}
                    +\frac12 \diff^2 \ubar\\
    \label{nu.23}
    \nu_{2,3}   &=  \frac{1}{128}
                        u \dbar^{-1}
                            \left(
                                |u|^2 \dbar^{-1} 
                                    \left(
                                        |u|^2 \dbar^{-1} (|u|^2) 
                                    \right)
                            \right)\\
                \nonumber
                & \quad 
                    -\frac{1}{32}
                        \left\{
                            \ubar \dbar^{-1} 
                                \left(
                                    u \diff 
                                        \left(
                                            \ubar \dbar^{-1}(|u|^2)
                                        \right)
                                \right) 
                                \right.\\
                &\qquad \qquad
                    \nonumber
                            \left.
                            + 
                            \ubar \dbar^{-1}
                                \left(
                                    |u|^2 \dbar^{-1}
                                        ( u \diff \ubar) 
                                \right) +
                            \diff   
                                \left(
                                    \ubar \dbar^{-1}
                                        \left(
                                            |u|^2 \dbar^{-1}(|u|^2)
                                        \right)
                                \right)
                        \right\}\\
                \nonumber
                &\quad  
                   + \frac18 
                        \left\{
                            u \dbar^{-1} (u \diff^2 \ubar) +
                            \diff^2 (u \dbar^{-1}(|u|^2)) +
                            \diff (\ubar \dbar^{-1}(u \diff \ubar))
                        \right)\} \\
                \nonumber
                &\quad  
                    - \frac12 \diff^3 \ubar
\end{align}

\subsection{\texorpdfstring{Expansions for $\nu_1^\sharp$ and $\nu_2^\sharp$}{Expansions for nu1-sharp and nu2-sharp}}
To obtain expansions for $\nu_1^\sharp$ and $\nu_2^\sharp$, since $\nu^\sharp$ is evaluated at $-z$, we can replace $u$ by {$\ubar$},\footnote{This corrects a sign error in \cite{Perry2016}, \S A.3} $\dbar^{-1}$ by $-\dbar^{-1}$, and $\diff$ by $-\diff$. This gives the following.

\medskip

\begin{align}
    \label{nusharp.start}
    \nu^\sharp_{1,0}    &=  {-}\frac14 \dbar^{-1}(|u|^2)  &
    \nu^\sharp_{2,0}    &=  {\frac12 u},\\[5pt]
    \label{nus12.1}
    \nu^\sharp_{1,1}    &=  \frac{1}{16} \dbar^{-1}(|u|^2 \dbar^{-1}(|u|^2))
                            -\frac14 \dbar^{-1}(\ubar \diff u) \\
    \nu^\sharp_{2,1}    &=  {-}\frac18 u \dbar^{-1}(|u|^2) { + \frac12 \diff u}
\end{align}
\begin{align}
\label{nus1.2}
\nu^\sharp_{1,2}    
    &=  {-}\frac{1}{64}
        \dbar^{-1}
        \left(
            |u|^2 \dbar^{-1}\left(|u|^2 \dbar^{-1}(|u|^2) \right)
        \right) \\
    \nonumber
    &\quad {+}\frac{1}{16}
        \left\{
            \dbar^{-1}
                \left(
                    \ubar \diff( u\dbar^{-1}(|u|^2))
                \right) +
            \dbar^{-1}
                \left(
                    |u|^2 \dbar^{-1}(\ubar \diff u)
                \right)
        \right\}
        {-} \frac14 \dbar^{-1}(\ubar \diff^2 u)
\\[10pt]
\label{nus2.2}
\nu^\sharp_{2,2}    &=  {+} \frac{1}{32}u \dbar^{-1}\left(|u|^2 \dbar^{-1}(|u|^2)\right) 
         { -} \frac18 \diff \left( u \dbar^{-1}(|u|^2) \right)
         { - } \frac18 u \dbar^{-1} (\ubar \diff u) {+ \frac12 \diff^2 u}\\[10pt]
\label{nus.23}
\nu^\sharp_{2,3}    &=
    { -}\frac{1}{128} u \dbar^{-1}
        \left(
            |u|^2 \dbar^{-1} 
                \left( |u|^2 \dbar^{-1}
                    (|u|^2)
                \right)
        \right)\\
    \nonumber
    & \quad { +} \frac{1}{32}
        \left\{
            u\dbar^{-1} 
                \left(
                    \ubar \diff
                        ( u \dbar^{-1}(|u|^2))
                \right)
            +   u \dbar^{-1}
                \left(
                    |u|^2 \dbar^{-1}
                        (\ubar \diff u)
                \right)
            +   \diff   
                \left(
                    u \dbar^{-1}
                        \left(
                            |u|^2 \dbar^{-1}(|u|^2)
                        \right)
                \right)
        \right\}\\
    \nonumber
    &\quad { -} \frac18
        \left\{
            u\dbar^{-1}
                \left(
                    \ubar \diff^2 u
                \right)
            + \diff^2
                \left(
                    u \dbar^{-1}(|u|^2)
                \right)
            + \diff 
                \left(
                    u \dbar^{-1}(\ubar \diff u )
                \right)
        \right\}\\
    \nonumber
    &\quad  {+}   \frac12 \diff^3 u
\end{align}

\bibliographystyle{amsplain}
\bibliography{mNV}

@incollection {Perry2019,
    AUTHOR = {Perry, Peter A.},
     TITLE = {Inverse scattering and global well-posedness in one and two
              space dimensions},
 BOOKTITLE = {Nonlinear dispersive partial differential equations and
              inverse scattering},
    SERIES = {Fields Inst. Commun.},
    VOLUME = {83},
     PAGES = {161--252},
 PUBLISHER = {Springer, New York},
      YEAR = {[2019] \copyright 2019},
      ISBN = {978-1-4939-9805-0; 978-1-4939-9806-7},
   MRCLASS = {35Q55 (35B30 35P25 35Q41 35R30)},
  MRNUMBER = {3931836},
MRREVIEWER = {Jason\ Carl\ Murphy},
}

@inproceedings{BC1988,
Author = {Beals, Richard and Coifman, R. R.},
Editor = {Degasperis, A and Fordy, A. P.  and Lakshmanan, M},
Title = {The Spectral Problem for the {D}avey-{S}tewartson and {I}shimori Hierarchies},
Booktitle = {Nonlinear Evolution Equations : Integrability and Spectral Methods},
Series = {Proceedings in Nonlinear Science},
Year = {1990},
Pages = {15-23},
Note = {Workshop on Nonlinear Evolution Equations : Integrability and Spectral
   Methods, Como, Italy, July 4-15, 1988},
ISBN = {0-7190-3273-3},
}

@article {BC1989,
    AUTHOR = {Beals, Richard and Coifman, R. R.},
     TITLE = {Linear spectral problems, nonlinear equations and the
              {$\overline\partial$}-method},
   JOURNAL = {Inverse Problems},
  FJOURNAL = {Inverse Problems. An International Journal on the Theory and
              Practice of Inverse Problems, Inverse Methods and Computerized
              Inversion of Data},
    VOLUME = {5},
      YEAR = {1989},
    NUMBER = {2},
     PAGES = {87--130},
      ISSN = {0266-5611,1361-6420},
   MRCLASS = {35Q20 (34A55 34B25 35P25 58F07 58G37)},
  MRNUMBER = {991913},
MRREVIEWER = {Alexander\ Yanovski},
       DOI = {10.1088/0266-5611/5/2/002},
       URL = {https://doi.org/10.1088/0266-5611/5/2/002},
}

@incollection {BC1984,
    AUTHOR = {Beals, R. and Coifman, R. R.},
     TITLE = {Multidimensional inverse scatterings and nonlinear partial
              differential equations},
 BOOKTITLE = {Pseudodifferential operators and applications ({N}otre {D}ame,
              {I}nd., 1984)},
    SERIES = {Proc. Sympos. Pure Math.},
    VOLUME = {43},
     PAGES = {45--70},
 PUBLISHER = {Amer. Math. Soc., Providence, RI},
      YEAR = {1985},
      ISBN = {0-8218-1469-9},
   MRCLASS = {35Q20 (35G20 35R30)},
  MRNUMBER = {812283},
MRREVIEWER = {Yusuke\ Kato},
       DOI = {10.1090/pspum/043/812283},
       URL = {https://doi.org/10.1090/pspum/043/812283},
}

@article {SungIII:1994,
    AUTHOR = {Sung, Li-Yeng},
     TITLE = {An inverse scattering transform for the {D}avey-{S}tewartson
              {II} equations. {III}},
   JOURNAL = {J. Math. Anal. Appl.},
  FJOURNAL = {Journal of Mathematical Analysis and Applications},
    VOLUME = {183},
      YEAR = {1994},
    NUMBER = {3},
     PAGES = {477--494},
      ISSN = {0022-247X,1096-0813},
   MRCLASS = {35Q55},
  MRNUMBER = {1274849},
MRREVIEWER = {Mehmet\ Can},
       DOI = {10.1006/jmaa.1994.1155},
       URL = {https://doi.org/10.1006/jmaa.1994.1155},
}

@article{NPT:2025,
    AUTHOR={Nachman, Adrian and Perry, Peter and Tataru, Daniel},
    TITLE={Global Well-Posedness for the modified {N}ovikov-{V}eselov Equation},
    YEAR=2025
}

@article {Perry2016,
    AUTHOR = {Perry, Peter A.},
     TITLE = {Global well-posedness and long-time asymptotics for the
              defocussing {D}avey-{S}tewartson {II} equation in
              {$H^{1,1}(\mathbb{C})$}},
      NOTE = {With an appendix by Michael Christ},
   JOURNAL = {J. Spectr. Theory},
  FJOURNAL = {Journal of Spectral Theory},
    VOLUME = {6},
      YEAR = {2016},
    NUMBER = {3},
     PAGES = {429--481},
      ISSN = {1664-039X,1664-0403},
   MRCLASS = {37K10 (32W05 35B30 35Q55 78A25 81Q05)},
  MRNUMBER = {3551174},
MRREVIEWER = {Dmitry\ E.\ Pelinovsky},
       DOI = {10.4171/JST/129},
       URL = {https://doi.org/10.4171/JST/129},
}

@article{Perry2014,
	author = {Perry, Peter A.},
	doi = {10.2140/apde.2014.7.311},
	fjournal = {Analysis \& PDE},
	issn = {2157-5045},
	journal = {Anal. PDE},
	mrclass = {35Q53 (35P25 35R30 37K15)},
	mrnumber = {3218811},
	mrreviewer = {Dmitry G. Shepelsky},
	number = {2},
	pages = {311--343},
	title = {Miura maps and inverse scattering for the {N}ovikov-{V}eselov equation},
	url = {https://doi-org.ezproxy.uky.edu/10.2140/apde.2014.7.311},
	volume = {7},
	year = {2014}}

@article{NRT2020,
	author = {Nachman, Adrian and Regev, Idan and Tataru, Daniel},
	doi = {10.1007/s00222-019-00930-0},
	fjournal = {Inventiones Mathematicae},
	issn = {0020-9910},
	journal = {Invent. Math.},
	mrclass = {42B10 (35Q55)},
	mrnumber = {4081134},
	mrreviewer = {Gaetano Siciliano},
	number = {2},
	pages = {395--451},
	title = {A nonlinear {P}lancherel theorem with applications to global well-posedness for the defocusing {D}avey-{S}tewartson equation and to the inverse boundary value problem of {C}alder\'{o}n},
	url = {https://doi-org.ezproxy.uky.edu/10.1007/s00222-019-00930-0},
	volume = {220},
	year = {2020}}

\end{document}